\documentclass{article}
\usepackage[utf8]{inputenc}
\usepackage[english]{babel}
\usepackage{graphicx}
\usepackage{amssymb}
\usepackage{amsmath}
\graphicspath{{images/}{../images/}}
\usepackage{amsthm}
\newtheorem{theorem}{Theorem}
\newtheorem{corollary}{Corollary}
\newtheorem{lemma}{Lemma}
\usepackage[T1]{fontenc}
\usepackage{scrextend}
\usepackage{mathtools}

\DeclarePairedDelimiter\floor{\lfloor}{\rfloor}
\newcommand\restr[2]{\ensuremath{\left.#1\right|_{#2}}}

\usepackage{blindtext}
 
\title{On the Intersection Number of Transformed Digraphs}
\author{Diljit Singh}
\date{ }
 
\begin{document}

\maketitle

\section*{Abstract} 

For any simple digraph $D$ we offer a new proof for the intersection number of its middle digraph, $M(D)$; while doing so we also solve for the intersection number when $D$ has loops. In addition, a new transformation, the union of $D$ and its subdivision digraph, is introduced and its intersection number calculated in full generality. For the total digraph, we extend previous arguments letting us solve for the intersection number of $T(D)$ with $D$ possibly having loops, but under the restriction loops of the digraph only touch (are to or from) themselves, sinks, and sources.

\section*{Introduction}
\indent A set representation of a graph is a family of sets with a particular rule which determines an edge in the graph. It is natural to take our vertices to be the sets, and our criteria for an edge to be the intersection of $2$ sets. In his seminal $1945$ work in intersection graph theory, Marczewski [4] formulated the above definition of the intersection graph along with proving every graph is the intersection graph of some sets. Given a graph, $G$, this begs the question: what is the minimal size of the set $U$ such that $G$ is the intersection graph of subsets of $U$? This is known as the intersection number of a graph. This was considered by Paul Erd\H{o}s, Adolph Winkler Goodman, and Louis Pósa [2] who in $1966$ proved that for any graph on $n$ vertices there exists a set with $\floor*{\frac{n^2}{4}}$ elements that is the set representation of $G$ and furthermore this is the smallest such number if the only information we are given is $G$ has $n$ vertices. This was proved by showing the intersection number is equal to the minimum number of complete subgraphs needed to cover $G$.

In this paper we consider the intersection number of digraphs, the theory of which has a number of analogous results to the above. If we let $U$ be a set with $\mathcal{F}$ a family of ordered pairs of subsets, with each ordered pair representing a vertex, $v$, and of the form $(S_v,T_v)$ then the intersection digraph of $\mathcal{F}$ has a vertex for every ordered pair and has an arc from $x$ to $y$ iff $S_x \cap T_y$ is not empty.

The intersection number of a digraph is defined as the minimum size of a set $U$ where our digraph, $D$, is the intersection digraph of ordered pairs of subsets of $U$. Thanks to work done in $1982$ by Beineke and Zamfirescu [1] and $1989$ by Sen et al. [5] it was shown that every digraph is the intersection digraph of ordered pairs of subsets of a set $U$. Sen et al. also introduced the concept of a generalized complete bipartite subdigraph (GBS), which are digraphs that have bipartite graphs (possibly with loops) as underlying graphs; furthermore, they must be subgraphs of the graph whose intersection number we are trying to find. Similar to the results of Erd\H{o}s and his coauthors, Sen et al. proved that the intersection number of a digraph, $D$, is equal to the minimum number of GBSs required to cover the arcs of $D$. 

The transformations we consider are the line digraph, subdivision digraph, middle digraph, and the total digraph. In addition we study the union of the subdivision digraph and the original digraph, we will call this $N(D)$. The motivation for these transformations lies in the adjacency structure of the original digraph $D$. These graphs are all defined in the next section.

Zamfirescu [6] proves a number of results concerning the intersection number of these transformations. Namely the intersection number of the line digraph and the subdivision digraph was completely solved for and results for the middle and total digraphs were given. In particular the intersection number of the middle and total digraphs was solved in the case that $D$ has no loops. For the middle digraph, we solve for the intersection number when $D$ has loops, we also provide a general construction to writing all minimal GBSs covers of $M(D)$ (along with all other transformations we consider). We also find the intersection number of $N(D)$ for $D$ possibly with loops. For the total digraph, we solve for the intersection number when $D$ has loops but limited to when every loop can only connect with, other than itself, only sources or sinks. 

\section*{Definitions}
    \begin{itemize}
        \item 
            A Generalized complete Bipartite Subdigraph(GBS) denoted: 
                    \[{x_1,x_2,\cdots,x_m} \choose {y_1,y_2,\cdots,y_n}\]  
            represents the set of $mn$ edges (or arcs, we use these two interchangeably) from each $x_i$ to each $y_i$. Note we can have $x_i=y_j$ for some $i$ and $j$. Alternatively we may denote it more compactly as: 
                    \[{\{X\}} \choose {\{Y\}}\] 
            where $\{X\}=\{{x_1,x_2,\cdots,x_m}\}$ and $\{Y\}=\{{y_1,y_2,\cdots,y_m}\}$.
        \item 
            Let $i(D)$ be the intersection number of $D$.
        \item 
            Denote the number of sinks and sources in a graph, $D$, by $A(D)$ and $B(D)$ respectively. Let $v(D)$ be the cardinality of the vertex set of $D$, and the number of loops in the graph $L_D$. If $D$ is understood then $A$, $B$, $v$, and $L$ are used. 
        \item
            Let $C(D)=v(D)-A(D)-B(D)$. Namely $C(D)$ is the number of vertices that are neither sinks nor sources. We will call these vertices \textbf{connectors.} Note loops are connectors. $C$ is used when $D$ is understood.
        \item The line digraph of $D$, $L(D)$, has vertex set $E(D)$, that is the edges of $D$. and an edge from $x$ to $y$ iff in $D$ the arc $x$ ends at the start of arc $y$.
        \item The subdivision digraph of $D$, $S(D)$, is the digraph obtained from placing a vertex on every arc of $D$, subdividing every arc into $2$ arcs which preserve the direction of the parent arc.
        \item The middle digraph of $D$, $M(D)$, is made from adding onto $S(D)$ the edges of $L(D)$.
        \item The total digraph of $D$, $T(D)$, is made from adding onto $M(D)$ the edges of $D$.
        \item 
            Let $D_L$ be the subgraph of $D$ where all loops are deleted.
        \item 
            Let $T^-(D)=N(D)+L(D)_L$. That is, $T^-(D)$ is the total digraph with loops from the line graph deleted.  
        \item 
            Let $\text{Loops}(D)$ be the graph s.t.
                    \[V(\text{Loops}(D))=\{v\in V(D) : (v,v)\in E\}\] 
            and 
                    \[E(\text{Loops}(D))=\{(v,v) : v\in V(\text{Loops}(D))\}.\]
        \item 
            Let $N(D)$ be the digraph that has vertex set $V(D)\cup E(D)$ (that is the vertex set of $D$ union the edge set of $D$) with $(a,b)$ an arc of $N(D)$ iff $(a,b)$ is an arc in $D$ or in $S(D)$. This transformation is very similar to the middle digraph of $D$ in the sense that if we add the edges of the line digraph we get $T(D)$.

\end{itemize}


\newpage


\section*{Main results}

\indent We begin by recalling a few relevant results. The first of which reduces finding the intersection number to finding the number of GBSs required to cover the graph.
\newline
\newline
\noindent\textbf{Theorem A.} (M. Sen, S. Das, A.B. Roy, D.B. West \text{[5]})\text{.} \textit{The intersection number of a digraph is the minimum number of GBSs required to cover its arcs.}
\newline 
\newline
\indent In $1964$ the notion of the Heuchenne condition (H-condition) was introduced [3]. A digraph satisfies the H-condition iff for every $u,v,w,x\in V(D)$, $vw,uw,ux\in E(D)$ implies $vx\in E(D)$ (our vertices need not be distinct). A characterization of line digraphs by the $H$-condition was also given then, namely it was proved that:
\newline
\newline
\noindent\textbf{Theorem B.} (Heuchenne \text{[3]})\text{.} \textit{$D$ is a line digraph iff it satisfies the H-condition.}
\newline 
\newline
\indent In her 2015 paper [6], Zamfirescu proved the subdivision digraph satisfies the $H$-condition. 
\newline
\newline
\noindent\textbf{Lemma C.} (Zamfirescu \text{[6]})\text{.} \textit{$S(D)$ satisfies the $H$-condition and has intersection number $2v-A-B$.}
\newline 
\newline
\indent In doing so it was established that $S(D)$ had a minimal unique cover, determined by the number of vertices, sources, and sinks, in the original digraph $D$.
\newline
\newline
\noindent\textbf{Lemma D.} (Zamfirescu \text{[6]})\text{.} \textit{If a digraph satisfies the $H$-condition the GBS cover is unique.}
\newline 

            The unique minimum set of GBSs for $S(D)$ is constructed by attaching to each vertex of $D$ a star at a sink, a star at a source, and $2$ GBSs at each connector: one for the incoming edges and one for outgoing. That is, for every vertex $x\in V(D)$ with incoming (resp. outgoing) arc set denoted $\{S_i\}$ (resp. $\{S_o\}$) every GBS associated to that vertex looks like:
                    
                    \[{\{S_i\}\choose x}, {x \choose \{S_o\}}.\] 
                    
            Note that if $x$ is a source (resp. sink) then $\{S_i\}=\emptyset$ (resp. $\{S_o\}=\emptyset$). In this case the former (resp. latter) GBS is non-existent and we only need to associate $1$ GBS to $x$. 
            
            This cover evolves into a minimum cover for $M(D)$ by the following:
                
                \begin{itemize}
                
                    \item
                        \textbf{If x does not have a loop:}
                    
                                \[{\{S_i\}\choose x} \rightarrow {\{S_i\}\choose x}\]
                            
                                \[{x \choose \{S_o\}} \rightarrow {\{S_i\},x \choose \{S_o\}}\]
                    \item 
                        \textbf{If x has a loop}, denote the loop by $x^2$, the GBSs below cover the added structure:
                                
                                \[{\{S_i\}\choose x} \rightarrow {\{S_i\}, x^2\choose x,x^2}\]
                                \[{x \choose \{S_o\}} \rightarrow {\{S_i\}, x, x^2 \choose \{S_o\}, x^2}.\]

                \end{itemize}

            Recall if any set, $\{S_i\}$ or $\{S_o\}$ is empty then we do not count that GBS.
            Alternatively by defining: $S^L_i=S_i\cup x^2$ and $S^L_o=S_i\cup x^2$ for every loop, $x^2$, we can re-write our GBSs in the case $x$ has a loop to:
    
                    \[{\{S^L_i\}\choose x, x^2}, {\{S^L_i\}, x \choose \{S^L_o\}}.\]

            This shows that to cover $M(D)$ we require for each vertex in $D$ a total of $2$ GBSs if both  $\{S_i\}$ and $\{S_o\}$ are non-empty and $1$ GBS when either $\{S_i\}$ or $\{S_o\}$ are empty. $\{S_i\}$ and $\{S_o\}$ are empty iff $x$ is a source in $D$ or a sink in $D$ respectively. So by this algorithm we require: 
                    \[2v(D)-A(D)-B(D)\] 
            GBS to cover $M(D)$. This proves: 
            
\begin{lemma}
\[i(M)\leq i(S).\]  
\end{lemma}

Note that $E(M(D))=E(S(D))\cup E(L(D))$ and $S(D)$ consists of edges defined from $D$ by the adjacency of arcs to vertex and vertex to arcs only, while $L(D)$ has only arc to arc edges. In particular we have no vertex to vertex edges in $M(D)$. Let $G$ be any fixed set of GBS coverings of $M(D)$. Then $g_i\in G$ has exactly one of the two forms: 
                    \[{A_i \choose V_i, A'_i},{A_i, V_i \choose A'_i}. \]
            Where $V_i$, $A_i$, and $A'_i$ are respectively sets of vertices, arcs, and arcs, in $D$. We now produce a covering of $S(D)$, denoted $\restr{G}{S}$, from our set $G$ by the following steps:
            
                \begin{enumerate}
                    
                    \item 
                        First we reduce the set $G$ to $G_1$. We do so by deleting every $g_i\in G$ that has empty $V_i$ set. We can do this because those GBSs don't intersect with $S(D)$. Note that $|G_1|\leq |G|$.
    
                    \item 
                        Since all vertex to arc or arc to vertex edges are from $S(D)$ and only $S(D)$, we can reduce every GBS, $g_i\in G_1$, to a reduced version denoted, $g'_i$. We do this by the below map:
                            
                            \begin{itemize}
                            
                                \item 
                                    If $g_i$ is of the form: $\displaystyle{A_i \choose V_i, A'_i}$ then $g'_i= \displaystyle{A_i \choose V_i}$.
                                \item 
                                    If $g_i$ is of the form: $\displaystyle{A_i, V_i \choose A'_i}$ then $\displaystyle g'_i= {V_i \choose A'_i}$.
                            \end{itemize}
                    
                    \item 
                        Let $G'$ be the set of reduced covers. Then we have $|G'| = |G_1|\leq |G|$.
                    \item 
                        Notice that $\restr{G}{S}=G'$. This is because it has exactly the edges of $M(D)$ that connect vertices of $D$ to or from from arcs of $D$. Those edges are exactly $S(D)$. So we have a GBS cover of $S(D)$ from a GBS cover of $M(D)$.

                \end{enumerate}

            If $i(M(D))$ is less than $2v(D)-A(D)-B(D)$ then we have $i(S(D))$ is less than $2v(D)-A(D)-B(D)$, but by lemma $C$ this is impossible. We have shown:
\begin{lemma} 
\[i(M)\geq i(S).\]     
\end{lemma}

By the above two results we get the intersection number of the middle transformation of a digraph, $D$. Namely it was shown: 

\begin{theorem}
\[M(D)=S(D)=2v-A-B.\]
\end{theorem}


We can further exploit the proof of lemma $2$ to get a lower bound of $2v-A-B$ on the intersection number of $N(D)$.

Note that $E(N(D))=E(S(D))\cup E(D)$ and $S(D)$ consists of edges defined from $D$ by the adjacency of arcs to vertex and vertex to arcs only, while $D$ has only vertex to vertex edges. In particular we have no arc to arc edges in $N(D)$. Let $G$ be a fixed set of GBS coverings of $N(D)$. Then $g_i\in G$ has exactly one of the two forms: 
                    \[{V_i \choose A_i, V'_i},{V_i, A_i \choose V'_i}. \]
            Where $V_i$, $V'_i$, and $A_i$ are respectively sets of vertices, vertices, and arcs, in $D$. We now produce a covering of $S(D)$, denoted $\restr{G}{S}$, from our set $G$ by the following steps:
             \begin{enumerate}
                    
                    \item 
                        First we reduce the set $G$ to $G_1$. We do so by deleting every $g_i\in G$ that has empty $A_i$ set. We can do this because those GBSs don't intersect with $S(D)$. Note that $|G_1|\leq |G|$.
    
                    \item 
                        Since all vertex to arc or arc to vertex edges are from $S(D)$ and only $S(D)$, we can reduce every GBS, $g_i\in G_1$, to a reduced version denoted, $g'_i$. We do this by the below map:
                            
                            \begin{itemize}
                            
                                \item 
                                    If $g_i$ is of the form: $\displaystyle{V_i \choose A_i, V'_i}$ then $g'_i= \displaystyle{V_i \choose A_i}$.
                                \item 
                                    If $g_i$ is of the form: $\displaystyle{V_i, A_i \choose V'_i}$ then $\displaystyle g'_i= {A_i \choose V'_i}$.
                            \end{itemize}
                    
                    \item 
                        Let $G'$ be the set of reduced GBSs. Then we have $|G'|= |G_1|\leq |G|$.
                    \item 
                        Notice that $\restr{G}{S}=G'$. This is because it has exactly the edges of $N(D)$ that connect vertices of $D$ to or from arcs of $D$. Those edges are exactly $S(D)$. So we have a GBS cover of $S(D)$ from a GBS cover of $N(D)$.

                \end{enumerate}

            If $i(N(D))$ is less than $2v(D)-A(D)-B(D)$ then we have $i(S(D))$ is less than $2v(D)-A(D)-B(D)$, but by lemma $C$, this is impossible. We have shown: 
            
\begin{lemma}
   \[i(N)\geq i(S).\]                 
\end{lemma}

Next, we solve for the intersection number of $N(D)$ and show it is exactly the sum of vertices in $D$ which are not sources added to the number of vertices which are not sinks. 
 
\begin{theorem}
\[N(D)=S(D)=2v-A-B.\]
\end{theorem}

\begin{proof}
For every vertex $x\in V(D)$ with incoming (resp. outgoing) arc set denoted $\{S_i\}$ (resp. $\{S_o\}$), every GBS in our minimum cover of $S(D)$ to that vertex looks like:
                    
                    \[{\{S_i\}\choose x}, {x \choose \{S_o\}}.\] 
                    
            Note that if $x$ is a source (resp. sink) then $\{S_i\}=\emptyset$ (resp. $\{S_o\}=\emptyset$). In this case the former (resp. latter) GBS is non-existent and we only need to associate $1$ GBS to $x$. 
            
            This cover evolves into a minimum cover for $N(D)$ by the below:
                \begin{itemize}
                    \item
                        If $x$ is a sink (resp. source) add the incoming (resp. outgoing) edges from $D$ to the respective GBS associated with $x$. This does not add any GBSs.
                    \item 
                        If $x$ is a carrier add the set of incoming edges to $\{S_i\}$ and outgoing edges to $\{S_o\}$. We again have not added any GBSs.
                \end{itemize}
            All extra edges from $D$ are accounted for and no extra GBSs were added to the set from $S(D)$. So we have proven $i(N(D))\leq i(S(D)).$

Along with lemma $3$ we are done.
\end{proof}


From now on, we deal with $D$ assuming it has loops, but vertices adjacent to all looped vertices are either that vertex itself, sources, or sinks. As always our $D$ has no multiple arcs. This simplification helps us prove a more general version of lemma $4$ in [1], stated below:
\newline
\newline
\noindent\textbf{Lemma E.} (Zamfirescu \text{[6]})\text{.} \textit{If $D$ has no loops and no multiple arcs, then no GBS of $T(D)$ can have two arcs from different GBSs of $S(D)$.}
\newline
\newline
\indent This is generally not true of $T(D)$, however we show that if we allow loops (touching sources or sinks), a weaker version is true for a subgraph of $T(D)$, namely $T^-(D)$ (recall that is the total digraph with loops from the line graph deleted). We then bound the worst-case error when constructing $T$ from $T^-$, giving a lower bound. Lastly, a cover achieving it is constructed showing equality.

The proof of this lemma also gives us the basis of a reduction principal when trying to calculate lower and upper bounds on the intersection number of the total digraph. Namely it tells us that the problem may be split into finding the intersection number of particular subgraphs and then gluing the subgraphs together. Simply adding the intersection number of the components gives an upper bound. Furthermore, the lemma establishes a lower-bound for the component of $T$ that arises from vertices of a certain distance from all loops.


\begin{lemma}
Two arcs from two different GBSs (in the minimal covering) of $S(D)$ can be in the same GBS of $T^-(D)$ iff both arcs contain the same vertex that has a loop. Furthermore, no GBS of $S(D)$ may vanish with this type of reduction.
\end{lemma}
\begin{proof}

The proof of lemma $E$ shows that we may not split up any GBS defined by vertices (each GBS is associated to a particular vertex and the cover is unique by lemmas $C$ and $D$) that are distance $2$ or more (in $D$). Now we consider the arcs made from sources and sinks that (in $D$) touch the loops or each other. Let $x,y$ be in the same GBS in $T^-$ but different GBSs in $S$. Begin by noting that we may not have any vertices that represent $D$-arcs that have more than $1$ incoming (resp. outgoing) edge from (resp. to) any vertices that represent $D$-vertices. Also in $S(D)$ we never have loops. In the below, arcs of $D$ are denoted by numbers, sources by $B_i$ for some $i$, and sinks by $A_i$ for some $i$, and a loop is $l$ (when needed its arc may be denoted $l^2$). 

\begin{itemize}

\item
First we consider the $2$ cases below:
    \[{{1}\choose{A_1}},{{2}\choose{A_2}}\]
    \[{{B_1}\choose{1}},{{B_2}\choose{2}}\]
Note in both cases $1\not= 2$ or else an arc vertex has at least two incoming or outgoing edges to 2 $D$-vertices. Now, if we were to adjoin the two arcs in $T^-$, we would have the same issue if $1=2$, namely an arc vertex, in $S(D)$, has at least two incoming or outgoing edges to $D$-vertices.

\item 

\[{{1}\choose{A}},{{B}\choose{2}}\]

We can't have $1=2$, which would imply a loop at an arc, but we have none of those in $T^-$. Clearly a source is not a sink, therefore $A\not=B$ and we need to use the loop vertex as an intermediary to have an arc to arc connection. However, if this is the case, we have an edge that goes from $2$ to $1$ opposite the (in this case) required edge that is from $1$ to $2$. Since we can not have both edges, this case is impossible.

\item 

We consider the $2$ cases below:
    \[{{1}\choose{l}},{{2}\choose{A}}\]
    \[{{l}\choose{1}},{{B}\choose{2}}\]

Whether $1$ is the same as $2$ or not, we have that if we can (in either case) join the 2 GBSs then an arc vertex has at least two incoming, or outgoing edges to 2 $D$-vertices. That is, it is needed for us to have $2$ to go into both $l$ and $A$ in the first case and $l$ and $B$ to go to $2$ in the second case. It was already established this can not happen in $T^-$.

\item 

\[{{l}\choose{1}},{{2}\choose{A}}\]

  Loops are not sources or sinks so $l\not = A$, and if we have $1= 2$, we have a loop at an arc, which cannot happen. So assume $1 \not =2$, in $D$, say $2$ is the arc from $v$ to $A$ and $1$ is the arc from $l$ to $w$. Since $l$ goes to $1$ we have $w$ is a sink. However, then we can never have (even if $w=A$ or if $v=l$) $2$ go to $1$.
    
\item 
    
\[{{1}\choose{l}},{{B}\choose{2}}\]
    
    Loops are not sources or sinks so $l\not = B$, and if we have $1= 2$, we have a loop at an arc, which cannot happen. So assume $1 \not =2$, in $D$, say $1$ is the arc from $w$ to $l$ and $2$ is the arc from $B$ to $v$. Since $1$ goes to $l$ we have $w$ is a source. However, then we can never have (even if $w=B$ or if $v=l$) $1$ go to $2$.

\item 

\[{{1}\choose{l}},{{l}\choose{2}}\]

If $1=2$ then we have a loop at an arc which can not happen. Now, in this case we can combine the two arcs. Consider the $2$ GBS in $S(D)$ associated with $l$, they are: \[{{\{S_i\},l^2}\choose{l}},{{l}\choose{\{S_o\},l^2}}.\] Notice that every vertex in $\{S_i\}$ goes to every vertex in $\{S_0\}$ and since in $T^-$ we have a loop at $l$, we can have arcs from the two GBS going into one another. However, since we have no loop at $l^2$ we always require there to be $2$ GBSs here: one with $l^2$ to $l$ the other with $l$ to $l^2$.

\end{itemize}
\end{proof}

\begin{corollary}
The intersection number of $T^-$ is at least the intersection number of $S$.
\end{corollary}
\begin{proof}
By the above lemma it is trivial.
\end{proof}


\begin{lemma}
The intersection number of $T^-(D)$ is less than or equal to that of $S(D)$.
\end{lemma}

\begin{proof} 

To cover $M(D)$ we need at most $2v(D)-A(D)-B(D)$ GBSs. We first expand the covers of $M$ that contained a looped vertex of $D$.
            The GBSs used for these structures in $M$ were of the form (for loop at $x$, denoted $x^2$):
                    \[{\{S^L_i\}\choose x, x^2}, {\{S^L_i\}, x \choose \{S^L_o\}}.\]
            We transform these GBSs by the below:
                    \[{\{S^L_i\}\choose x, x^2} \rightarrow {\{S_i\},x\choose \{S_o\},x,x^2}\]
                    and
                    \[{\{S^L_i\}, x \choose \{S^L_o\}} \rightarrow {x^2\choose x, \{S_o\}.}\]

        In addition to all of $M$ minus the loops of $L(D)$ we have covered $\text{Loops}(D)$ and have used $2v-A-B$ GBSs to do so. We still need to cover $D_L$. We show to cover the remaining edges of $D$ is free, that is we cover $D_L$ only by expanding existing GBSs. 
            
            First consider all sources or sinks, in our construction of $M(D)$, every source or sink has associated to it a star, we simply grow this star with any extra edges in our graph added to it. This leaves only connector to connector edges from $D_L$ to be covered. In $M(D)$ every non-loop connector in $D$, call it $y$, has $2$ GBSs attached to it:
            \[{\{S_i\}\choose y}, {\{S_i\},y \choose \{S_o\}}.\]
            \indent Notice the first GBS is a star with a vertex from $D$ in its center. Stars are always free to expand and we can simply add any connector to connector edges from $D_L$, to the star with incoming edges associated to one of the connectors. That is, our GBS at a non-loop connector, $y$, which has incoming vertices $v\in V_y$ in $T^-$ is:
            
        \[{\{S_i\},\{V_y\}\choose y}, {\{S_i\},y \choose \{S_o\}}.\]
            So we have shown:
                    \[i(T^-)\leq 2v-A-B.\] 

\end{proof}


\begin{lemma}
The intersection number of $T^-(D)$ is $2v-A-B$. 
\end{lemma}
\begin{proof}
By the above lemma and corollary this is obvious.
\end{proof}


We wish to study the possible GBSs of $T^-$ which, when we add loops of the line digraph, have arcs from $S(D)$ which may be in the same GBS as arcs from other GBSs of $S(D)$, or in the case of a looped vertex, GBS from $S(D)$ which may be able to disappear. That is, we wish to study the possible GBSs which may be either grown (possibly from nothing) or reduced (possibly to nothing). We will bound the number of said GBSs.

\begin{lemma}
\[i(T)\leq i(T^-)-L .\]
\end{lemma}
\begin{proof}

We focus on a single $D$-vertex loop and bound the error (how many GBSs we can reduce) that comes with adding a loop to its arc-vertex, then we sum the error over for all loops. We can do this because (in $D$) our loops only touch sources or sinks, while vertices in any GBS must have at most distance $2$ from each other. The extra vertices from $S(D)$ makes it so one GBS may not have two $D$-vertex loops (or loop related structure) in it. Furthermore, our proof of lemma $4$ makes it so that adding one loop, $l$, to $T^-$ at a loop born arc will not affect any GBS not touching $l$ or $l^2$.

For any cover, $G$, of $T^-$, let $R=\{r_i\}_{i=1}^n$ be the GBSs which can be reduced or deleted entirely when we add the loop. Note that since our graphs do not generally satisfy the $H$-condition or other nice properties, the GBS cover need not be unique. We define deletion (and resp. growth) to be when an arc from $S(D)$ is separated from (resp. added to) a GBS containing another arc from the same GBS in $S(D)$. Deletion can also be when a whole GBS ceases to exist; this is also considered a reduction. The creation of a brand new GBS is also counted as growth. This covers cases where GBSs have only $1$ edge or the cases of non-optimal coverings of $T^-$ being reduced. We call edges in $S(D)$ primitive. New-primitive edges are when the two primitive edges are from different GBS in $S$ but the same in $T$.

Let $G=\{g_i\}_{i=1}^m$ be the GBSs which are grown with the addition of the loop and at least $1$ new-primitive edge from an $r\in R$. These will be the GBSs which contain $2$ arcs from different GBSs from $S$. They must contain the added loop by our proof of lemma $4$. Since if, say, $g_i$ does not need the loop to be added to it before a new-primitive edge from an $r_j$ is, our cover of $T^-$ may be altered to still be a cover but also contain arcs from different GBSs of $S(D)$ in one GBS, namely in $g_i$, or in the case of the loop vertex GBS, the GBS may simply cease to exist. 

Let $r_i\in R$ be $\binom{\{A_i\}}{\{B_i\}}$ and $g_j\in G$ be $\binom{\{X_j\}}{\{Y_j\}}$. After we add the loop we grow $g_j$ from $r_i$ to get:
\[\binom{\{X_j\}}{\{Y_j\}} \rightarrow \binom{\{X_j\},l,m_i}{\{Y_j\},l,n_i}\]
with $m_i \subseteq A_i$ and $n_i \subseteq B_i$. The existence of $\binom{l}{\{Y_q\},n_s}$ and $\binom{\{X_q\},m}{l}$ implies that if we are reducing an optimal covering of $T^-$, then by a combination of lemma $4$, and the fact after the loop is added all edges touching both the loop from the line digraph and loop from the original digraph can be covered in $1$ GBS, we have that $m+n=2$, at least that is the case if our original covering (of $T^-$) was minimal. This can be seen by considering our construction of the covering sets of $T^-$ (in the proof of the upper bound). In this case, simply assume we delete all the GBSs which are reduced (that is, subtract $1$ from the GBS count for every loop added) to get the lower bound of $i(T^-)-L$ for $i(T)$.  

If we have $n>1$,  then we are using a suboptimal cover of $T^-$. Assume we can, for any fixed loop, result in a different lower bound than above. Let $m+n=k$ and let $I$ be the number of GBSs in this covering. We only care about this if we have $I$ is less than $G-1$ when we delete as many of the $r_i$ as possible, with the worst case being that we delete $n=k-m$ GBSs. This gives us $I-k+m \leq G-1 $. Note that implies $I \leq G+k-m-1$. However, $I$ was not optimal and by lemma $4$ we require at least $k-m-1$ extra GBS for this covering than the optimal case. This can be seen because in this case, $n=k-m$ and $n=1$ is the realized worst case (our proof of the intersection number of $T^-$) so at least $k-m-1$  more GBS than the minimum are needed to exist in this case. Now we have $G \leq I-k+m+1$. Proving in a suboptimal covering we can, at best, do the same as the optimal covering  case done above. Thus we have shown:

\[i(T)\leq i(T^-)-L.\]

\end{proof}


Next, we solve for the intersection number of $T(D)$ and show it is exactly the sum of the vertices in $D$ which are not sources added to the number of vertices which are not sinks minus the number of loops.


\begin{theorem} For $D$ with loops only touching sources or sinks, we have: $T(D)=2v-A-B-L$.
\end{theorem}
\begin{proof}

 To cover $M(D)$ we need at most $2v(D)-A(D)-B(D)$ GBSs, in particular every loop needs $2$ GBS to cover it. However when we add in $\text{Loops}(D)$ we only need $1$ GBS to cover the structure at that loop. This can be seen in the below contraction of the GBS we made for $M(D)$.
            Our GBS of
                    \[{\{S^L_i\}\choose x, x^2}, {\{S^L_i\}, x \choose \{S^L_o\}}.\]
            Becomes:
                    \[{\{S_i\}, x, x^2\choose  \{S_o\}, x, x^2}.\]
            Or expressed in terms of $S^L_i$ and $S^L_o$ it is
                    \[{\{S^L_i\}, x\choose  \{S^L_o\}, x}.\]

        From $D$, we have covered the loops and have used $2v-A-B-L$ GBSs to do so. We still need to cover $D_L$. Consider all sources or sinks, in our construction of $M(D)$, every source or sink has associated to it a star, we simply grow this star with any extra edges added. This leaves only connector to connector edges from $D_L$ to be covered. In $M(D)$ every non-loop connector in $D$, call it $y$, has $2$ GBSs attached to it:\[{\{S_i\}\choose y}, {\{S_i\},y \choose \{S_o\}}.\]
            \indent Notice the first GBS is a star with a vertex from $D$ in its center. Stars are always free to expand and we can simply add any connector to connector edges from $D_L$ to the star with incoming edges associated to one of the connectors. That is, our GBS at a non-loop connector, $y$, which has incoming vertices $v\in V_y$ in $T$ is:
            
        \[{\{S_i\},\{V_y\}\choose y}, {\{S_i\},y \choose \{S_o\}}.\]
            So we have shown:
                    \[i(T(D))\leq 2v(D)-A(D)-B(D)-L.\] 

In addition lemma $7$ establishes a lower bound that is equal to the upper.

\end{proof}

\newpage
\section*{References}

\begin{enumerate}
    \item L.W. Beineke, C. M. Zamfirescu: Connection digraphs and second order line digraphs, Discrete Math 39 (1982) 237-254.
    \item P. Erd\H{o}s, A. Goodman, L. Posa: The representation of a graph by set intersections, Can. J. Math 18 (1966) 106$-$112.
    \item C. Heuchenne: Sur une certaine correspondence entre graphs, Bull. Soc. Roy. Lige 33 (1964) 743-753.
    \item E. Marczewski: Sur deux proprits des classes d'ensembles, Fund. Math 33 (1945) 303-307.
    \item M. Sen, S. Das, A.B. Roy, D.B. West: Interval digraphs - An analogue of interval graphs, J. Graph Th. 13 (1989)
    189-202.
   \item C. M. Zamfirescu: Transformations of Digraphs Viewed as Intersection Digraphs, Convexity and Discrete Geometry Including Graph Theory, eds: K. Adiprasito, I.B\'ar\'any, and C.Vilcu, Springer, 2016.

\end{enumerate}
\end{document}